\documentclass[12 pt]{article}%
\usepackage{amsmath, amsfonts, amsthm, color,latexsym}
\usepackage{amsmath, ulem}
\usepackage{amsfonts}
\usepackage{amssymb}
\usepackage{color, soul}
\usepackage[all]{xy}
\usepackage{graphicx}%
\providecommand{\U}[1]{\protect\rule{.1in}{.1in}}
\allowdisplaybreaks[4]
\newtheorem{theorem}{Theorem}[section]
\newtheorem{proposition}[theorem]{Proposition}
\newtheorem{corollary}[theorem]{Corollary}
\newtheorem{example}[theorem]{Example}
\newtheorem{examples}[theorem]{Examples}
\newtheorem{remark}[theorem]{Remark}

\newtheorem{lemma}[theorem]{Lemma}
\newtheorem{final remark}[theorem]{Final Remark}
\newtheorem{definition}[theorem]{Definition}
\allowdisplaybreaks[4]

\newcommand{\cvf}{\overset{\omega}{\rightarrow}}

\newcommand {\cvfe} {\overset{\omega^\ast}{\rightarrow}}
\newcommand {\R}{\mathbb{R}}
\newcommand {\K} {\mathbb{K}}

\newcommand {\N} {\mathbb{N}}
\newcommand{\norma}[1]{\| #1 \|}
\newcommand{\sol}[1]{\text{sol}(#1)}

\begin{document}

\title{Disjoint Dunford-Pettis-type properties in Banach lattices}
\author{Geraldo Botelho\thanks{Supported by Fapemig grants PPM-00450-17,  RED-00133-21 and APQ-01853-23.}\,, Jos\'e Lucas P. Luiz\thanks{{Supported by Fapemig grant APQ-01853-23}}\, and  Vinícius C. C. Miranda\thanks{Supported by CNPq Grant 150894/2022-8 and Fapemig grant APQ-01853-23\newline 2020 Mathematics Subject Classification: 46B42, 46B40, 46G25, 46M05.\newline Keywords: Banach lattices, Dunford-Pettis properties, regular homogeneous polynomials. }}
\date{}
\maketitle

\begin{abstract}
New characterizations of the disjoint Dunford-Pettis property of order $p$ (disjoint $DPP_p$) are proved and applied to show that a Banach lattice of cotype $p$ has the disjoint $DPP_p$ whenever its dual has this property. The disjoint Dunford-Pettis$^*$ property of order $p$ (disjoint $DP^*P_p$) is thoroughly investigated. Close connections with the positive Schur property of order $p$, with the disjoint $DPP_p$, with the $p$-weak $DP^*$ property and with the positive $DP^*$ property of order $p$ are established. In a final section we study the polynomial versions of the disjoint $DPP_p$ and of the disjoint $DP^*P_p$.
\end{abstract}

\section{Introduction}

Throughout this paper, $X$ and $Y$ denote Banach spaces,  $E$ and $F$ denote Banach lattices, $B_X$ and $S_X$ denote the closed unit ball and the unit sphere of $X$, and $p$ is a real number not smaller than $1$.

 Recall that a subset $A$ of a Banach space $X$ is a Dunford-Pettis set (respectively, a limited set) if every weakly null (respectively, weak* null) sequence in the dual $X^\ast$ of $X$ converges uniformly to zero on $A$. A Banach space $X$ has the (classical) Dunford-Pettis property ($DPP$) if weakly compact subsets of $X$ are Dunford-Pettis (see \cite[p.\,36]{andrews}), and $X$ has the Dunford-Pettis$^*$ property ($DP^*P$) if weakly compact subsets of $X$ are limited (see \cite[Definition 1.1]{cargalou}). Lattice counterparts of these concepts have been studied recently: a bounded subset $A$ of a Banach lattice $E$ is almost Dunford-Pettis (respectively, almost limited) if every disjoint weakly null (respectively, disjoint weak$^*$ null) sequence in $E^\ast$ converges uniformly to zero on $A$ (see \cite{bouras}, respectively, \cite{chen}). The Banach lattice $E$ has the weak Dunford-Pettis property ($wDPP$) if weakly compact subsets of $X$ are almost Dunford-Pettis, and $E$ has the weak Dunford-Pettis$^*$ property ($wDP^*P$) if weakly compact subsets of $E$ are almost limited (see \cite[Theorem 2.7]{bouras} and \cite[Definition 3.1]{chen}, respectively).

   According to Alikhani \cite[Definition 4.1]{ali}, a Banach lattice $E$ has the disjoint Dunford-Pettis property or order $p$ (disjoint $DPP_p$), if every disjoint weakly $p$-summable sequence is Dunford-Pettis. In Section 2 of this paper we continue the study of the disjoint $DPP_p$. First we give some new characterizations of this property that shall be useful later. Then we show that, unlike the cases of the $DPP$ and of the $wDPP$, it is not true that $E$ has the disjoint $DPP_p$ whenever $E^*$ has this property. In the search for an additional condition under which this implication holds true, in Theorem \ref{lb5r} we prove that if $E^*$ has the disjoint $DPP_p$ and $E$ has cotype $p$, then $E$ has the disjoint $DPP_p$.

  In the same way that the $DP^*P$ is defined by replacing Dunford-Pettis sets in the $DPP$ with limited sets, in Section 3 we introduce the disjoint Dunford-Pettis$^*$ property of order $p$ (disjoint $DP^*P_p$) by making the same replacement in the definition of the disjont $DPP_p$. Illustrative examples are given and this new property is compared to several related properties already studied. Denoting the positive Schur property of order $p$ introduced in \cite{zeefou} by  $SP_p^+$, in particular we see that the following implications are always true and that their converses do not hold in general:
  $$ SP_p^+ \Longrightarrow \text{ disjoint } DP^*P_p \Longrightarrow  \text{ disjoint } DPP_p. $$ Some effort is devoted to find additional conditions under which the converse implications hold. As to the converse of the first implication, the order continuity of the norm does the job (cf. Proposition \ref{hgty}), whereas the weak Grothendieck property does the job for the second implication in $\sigma$-Dedekind complete Banach lattices (cf. Corollary \ref{cor1}). Strong connections of the disjoint $DP^*P_p$ with the $p$-weak $DP^*$ property from \cite{ardakani} and with the positive $DP^*$ property of order $p$ from \cite{ardakani2} are also established.

  The interest in studying polynomial versions of geometric and topological properties of Banach spaces has increased in the last two decades. For instance,  polynomial versions of the $DPP$ and of the $DP^\ast P$ can be found in \cite{cargalou, farjohn, pel, ryan}. In the context of Banach lattices, the positive polynomial Schur property was investigated in \cite{jlucaspams} and polynomials versions of the $wDPP$ and of the $wDP^\ast P$ appeared in \cite{wangshibu}. In Section 4 we investigate the polynomial behavior of the disjoint $DPP_p$ and of the disjoint $DP^*P_p$. We show that the characterizations of these properties we proved in the previous sections hold true if linear operators/functionals are replaced with homogeneous polynomials. In other words, we prove that the polynomial versions of these properties coincide with themselves.%  The result is that these two properties can be characterized    , we give the polynomial versions of the properties studied in Section 2.

We refer the reader to \cite{alip, meyer} for background on Banach lattices, to \cite{fabian} for Banach space theory,  and to \cite{dineen} for polynomials on Banach spaces.

Whenever we say that a sequence $(x_n)_n$ has a property defined for sets we mean that the set $\{x_n: n \in \N\}$ has the property.

We finish this introduction stating some useful and inspiring characterizations of the properties considered above. \\
%In the context of Banach lattices, two further classes of sets were introduced by considering disjoint sequences. Following \cite{bouras} (resp. \cite{chen}), a bounded subset $A$ of a Banach lattice $E$ is said to be almost Dunford-Pettis (resp. almost limited) if every disjoint weakly null (resp. disjoint weak* null) sequence in $E^\ast$ converges uniformly to zero on $A$. Throughout this paper, we say that $(x_n)_n$ is a (an almost) Dunford-Pettis/limited sequence, if $\{ x_n : n \in \N \}$ is a (an almost) Dunford-Pettis/limited set.
%In the following, we list four properties that relates weakly compact subsets and the classes of sets mentioned in the former paragraph:
\noindent$\bullet$ A Banach space $X$ has the $DPP$ if and only if $x_n^\ast(x_n)\longrightarrow 0$ for every weakly null sequence $(x_n)_n$ in $X$ and every weakly null sequence $(x_n^\ast)_n$ in $\subset X^\ast$ \cite[Proposition 2]{castgon}.\\
\noindent$\bullet$ A Banach space $X$ has the $DP^*P$ if and only if $x_n^\ast(x_n)\longrightarrow 0$ for every weakly null sequence $(x_n)_n$ in $ X$ and every weak* null sequence $(x_n^\ast)_n$ in $X^\ast$. \cite[p.\,206]{cargalou}.\\
\noindent$\bullet$ A Banach lattice $E$ has the $wDPP$ if and only if $x_n^\ast(x_n)\longrightarrow 0$ for every weakly null sequence $(x_n)_n$ in $ E$ and every disjoint weakly null sequence $(x_n^\ast)_n$ in $ E^\ast$ \cite[Theorem 2.7]{bouras}.\\
\noindent$\bullet$ A Banach lattice $E$ has the $wDP^*P$ if and only if $x_n^\ast(x_n)\longrightarrow 0$ for every weakly null sequence $(x_n)_n$ in $E$ and every disjoint weak* null sequence $(x_n^\ast)_n$ in $ E^\ast$ \cite[p.\,551]{chen}.

\section{The disjoint Dunford-Pettis property of order $p$}

Extending the linear concept introduced in \cite{zeefou}, a (linear or nonlinear) map from a Banach lattice to a Banach space is said to be disjoint $p$-convergent if it sends disjoint weakly $p$-summable sequences to norm null sequences. Recent developments can be found in \cite{ali, ardakani2, ardakani, botgarmir}. The positive Schur property of order $p$ ($SP_p^+$) and the disjoint $DPP_p$ can be characterized by means of disjoint $p$-convergent operators as follows:\\
\noindent$\bullet$ A Banach lattice has the $SP_p^+$ if and only if the identity operator is disjoint $p$-convergent, that is, if disjoint weakly $p$-summable sequences are norm null \cite[Remark 4.7]{zeefou}.\\
\noindent$\bullet$ A Banach lattice $E$ has the disjoint $DPP_p$ if and only if every weakly compact operator $T\colon E \to c_0$ is disjoint $p$-convergent %, meaning that $T$ sends disjoint weakly $p$-summable sequences to norm null sequences (see
 \cite[Theorem 4.2]{ali}.

\begin{examples}\rm \label{obs1}
 (1) \rm It is easy to see that every Banach lattice $E$ with the $DPP$ or with the  $SP_p^+$ has the disjoint $DPP_p$. \\
 (2) The Banach lattice $\ell_2$ does not have the disjoint $DPP_2$.  Indeed, the sequence $(e_n)_n$ of the canonical unit vectors is  disjoint and weakly $2$-summable, but $\{ e_n : n \in \N \}$ is not a Dunford-Pettis subset of $\ell_2$. \\
 (3) Let $1 < q \leq p$ be given. Since weakly $q$-summable sequences are weakly $p$-summable, every Banach lattice with the disjoint $DPP_p$ has the the disjoint $DPP_q$ as well. %Indeed, if $(x_n)_n$ is a disjoint weakly $q$-summable sequence, then it is disjoint weakly $p$-summable, and then $\{ x_n : n \in \N \}$ is a Dunford-Pettis subset of $E$. Therefore $E$ has the disjoint $DPP_q$.
 Let us see that the converse is not true. As we saw above, $\ell_2$ fails the disjoint $DPP_2$, but $\ell_2$ has the disjoint $DPP_1$ by (1) because it has the $SP_1^+$ (see \cite[Example 3.9(ii)]{zeefou}). \\
 (4) The first example of a Banach space with the $DPP$ whose dual fails the $DPP$ was $E = \ell_1(\bigoplus_{n=1}^\infty \ell_2^n)$ (see \cite{stegall}). In the Banach lattice context, it was proved in \cite{loumir} that $E^\ast = \ell_\infty(\bigoplus_{n=1}^\infty \ell_2^n)$ even fails the $wDPP$. Moreover, it is known that $E^\ast % = \ell_\infty(\bigoplus_{n=1}^\infty \ell_2^n)
 $ contains a complemented  lattice copy of $\ell_2$ (see the proof of \cite[Proposition 2.4]{loumir}), hence, by \cite[Corollary 4.3(i)]{ali}, we conclude that
$E^\ast %= \ell_\infty(\bigoplus_{n=1}^\infty \ell_2^n)
$ fails the $DPP_2$.
\end{examples}

It was proved in \cite[Proposition 5.5]{galmir} that every Dunford-Pettis disjoint sequence in a Banach lattice is almost Dunford-Pettis. %if a disjoint sequence $(x_n)_n$ in a Banach lattice $E$ is almost Dunford-Pettis, then it is Dunford-Pettis.
Therefore, $E$ has the disjoint $DPP_p$ if and only if every disjoint weakly $p$-summable sequence in $E$ is almost Dunford-Pettis. Next we establish some further characterizations that shall be useful later a couple of times. We include the aforementioned characterization just for the record.

%%\noindent \textbf{Question 1: } If $E^\ast$ has the $DPP_p$, does $E$ has the $DPP_p$? \textcolor{red}{Ver Example \ref{ex1}.}

%\smallskip

%%\noindent \textbf{Question 2: } It is well known that $E = \ell_1(\ell_2^n)$ does has the $DPP$ (hence the $wDPP$), but $E^\ast$ does not have the $DPP$ or the $wDPP$ (see \cite{loumir}). Does $E^\ast$ has the disjoint $DPP_p$? for some $1 \leq p < \infty$? \textcolor{red}{Ver Example \ref{ex2}.}

%%\smallskip

%%It was proved in \cite[Proposition 5.5]{galmir} that if a disjoint sequence $(x_n)_n$ in a Banach lattice $E$ is almost Dunford-Pettis, then it is Dunford-Pettis. Therefore, $E$ has the disjoint $DPP_p$ if and only if every disjoint weakly $p$-summable sequence in $E$ is almost Dunford-Pettis.

\begin{theorem} \label{teorema1}
    The following are equivalent for a Banach lattice $E$. \\
    {\rm (a)} $E$ has the disjoint $DPP_p$. \\
    {\rm (b)} Every disjoint weakly $p$-summable sequence in $E$ is almost Dunford-Pettis. \\
    {\rm (c)} For every disjoint weakly $p$-summable sequence $(x_n)_n$ in $E$ and every disjoint weakly null sequence $(x_n^\ast)_n$ in $E^\ast$, we have  $x_n^\ast(x_n) \longrightarrow 0$. \\
    {\rm (d)} For every disjoint weakly $p$-summable sequence $(x_n)_n$ in $ E$ and every positive disjoint weakly null sequence $(x_n^\ast)_n$ in $ E^\ast$, we have $x_n^\ast(x_n) \longrightarrow 0$.\\
    {\rm (e)} For every disjoint weakly $p$-summable sequence $(x_n)_n$ in $ E$ and every positive weakly null sequence $(x_n^\ast)_n$ in $E^\ast$, we have $x_n^\ast(x_n) \longrightarrow 0$.
\end{theorem}

\begin{proof}
    As mentioned above, (a) $\Leftrightarrow$ (b) follows from \cite[Proposition 5.5]{galmir}. The implications (c) $\Rightarrow$ (d) and (e) $\Rightarrow$ (d) are immediate. The implication (a) $ \Rightarrow$ (e) follows from \cite[Theorem 4.2]{ali}. We just have to check (b) $\Rightarrow$ (c) and (d) $\Rightarrow$ (b).

    (b) $\Rightarrow$ (c) Let $(x_n)_n$ be a disjoint weakly $p$-summable sequence in $E$ and let $(x_n^\ast)_n$ be a disjoint weakly null sequence in $E^\ast$. By the assumption the set $ \{ x_k : k \in \N \}$ is almost Dunford-Pettis, so
    $$ |x_n^\ast(x_n)| \leq \sup_{k \in \N} |x_n^\ast (x_k)| \longrightarrow 0. $$

    (d) $\Rightarrow$ (b) Let $(x_n)_n$ be a disjoint weakly $p$-summable sequence in $E$ and let $A = \{ x_n : n \in \N \}$.
    Suppose that $A$ is not an almost Dunford-Pettis set. In this case there exist
    $\varepsilon > 0$, a sequence $(y_{k})_k \subset A$ and a disjoint weakly null sequence $(x_k^\ast)_k \subset E^\ast$ such that $ |x_k^\ast| (|y_{k}|) \geq |x_k^\ast (y_{k})| \geq \varepsilon$ for every $k \in \N$.
    Since finite sets are compact, hence almost Dunford-Pettis, $\{ y_{k} : k \in \N \}$ cannot be finite, so we can assume without loss of generality that $(y_{k})_k$ is a subsequence of $(x_n)_n$, say $(y_{k})_k= (x_{n_k})_k$. However, since $(|x_{n_k}|)_k$ is a positive disjoint weakly null sequence (see \cite[Proposition 1.3]{wnukdual}), it follows that $|x_k ^\ast|(|y_{k}|) =|x_k ^\ast|(|x_{n_k}|) \longrightarrow 0$, a contradiction that completes the proof.
   %%% \textcolor{blue}{Since finite sets are Dunford-Pettis, $\{ x_{k_n}: n \in \N \}$ and so we can assume without loss of generality that $(x_{k_n})_n$ is a subsequence of $(x_k)_k$. (TÁ MEIO CONFUSO ESSA DEM. PRECISAMOS REVER!)}  Nevertheless, since $(|x_{k_n}|)_n$ is a disjoint weakly $p$-summable sequence (see \cite[Proposition 2.2]{zeefou}) and since $(|x_n^\ast|)_n$ is a positive disjoint weakly null sequence (see \cite[Proposition 1.3]{wnukdual}), we would have that $|x_n^\ast|(|x_{k_n}|) \to 0$, a contradiction.
  %   and the implication (e) $\Rightarrow$ (d) is immediate.
\end{proof}

It is well known that if $X^\ast$ has the $DPP$ (resp. $E^\ast$ has the $wDPP$), then $X$ has the $DPP$ (resp. $E$  has the $wDPP$) (see \cite[Corollary 3]{castgon}, resp. \cite[Corollary 2.11]{boumoussa}). In the following example we show that a related result does hold in general for the disjoint $DPP_p$.

\begin{example} \label{exemplo1}\rm
On the one hand,    the Banach lattice $\ell_3$ does not have the disjoint $DPP_{3/2}$ because $(e_n)_n$ is a disjoint weakly $\frac{3}{2}$-summable sequence and the sequence $(e_n^\ast)_n$ of coordinate functionals is a disjoint and weakly null sequence $\ell_3^\ast$ such that $e_n^\ast(e_n) = 1$ for every $n \in \N$. On the other hand, $\ell_{3}^\ast = \ell_{3/2}$ has the disjoint $DPP_{3/2}$ by Example \ref{obs1}(1) because it has the $SP_{3/2}$ (see \cite[Example 3.9(i)]{zeefou}).
\end{example}

Now it is a natural question to ask what condition should be added to the disjoint $DPP_p$ of $E^*$ in order to ensure that $E$ has the disjoint $DPP_p$.
%Nevertheless, whenever $E^\ast$ has the disjoint $DPP_p$ and type $p^\ast$, then $E$ has the disjoint $DPP_p$. Before we prove this fact,
Recall that a Banach space $X$ has {\it cotype} $p \geq 2$ if there is a constant $C > 0$ such that for every finite subset $\{x_1, \dots, x_n \} \subset X$,
$$ \left ( \sum_{k=1}^n \norma{x_k}^p \right )^{1/p}\leq C \cdot \left ( \int_0^1 \left\|\sum_{k=1}^n r_k(t) x_k\right\|^2 \right )^{1/2} , $$
where $(r_k)_k$ denotes the Rademacher sequence (see \cite[p.\,217]{diestel}).% or \cite[Definition 6.2.10]{albiac}).
%A Banach space $X$ is said to have {\it nontrivial type} if it has some type $1 < q \leq 2$ \textcolor{blue}{(Precisa dessa defini\c c\~ao? Acho que n\~ao)}.
%In \cite[Lemma 3.4]{fouzee}, Fourie and Zeekoei proved that in a Banach lattice $E$ with type $1< q \leq 2$, every disjoint sequence in the solid hull of a relatively weakly compact subset of $E$ is weakly $p$-summable for all $p \geq q^\ast$. We will use this result to prove the following:

\begin{proposition}\label{lb5r}
    If $E$ has cotype $p$ and $E^\ast$ has the disjoint $DPP_p$, then $E$ has the disjoint $DPP_p$.
\end{proposition}

\begin{proof}
    Let $(x_n)_n$ be a disjoint weakly $p$-summable sequence in $E$ and let $(x_n^\ast)_n$ be a disjoint weakly null sequence in $E^\ast$. Since $\{ x_n^\ast : n \in \N \}$ is a weakly compact subset of $E^\ast$, $(x_n^\ast)_n$ is a disjoint sequence contained in the solid hull of a weakly compact set. The cotype $p$ of $E$ implies, by the principle of local reflexivity, that $E^{**}$ has cotype $p$ as well (see \cite[Corollary 11.9]{diestel}). From \cite[Lemma 2.4]{botgarmir}, which is a straightforward consequence of the proof of \cite[Lemma 3.4]{fouzee}, it follows that $(x_n^\ast)_n$ is a  disjoint weakly $p$-summable sequence in $E^\ast$. As the canonical embedding $J \colon E \to E^{\ast\ast}$ is a lattice homomorphism \cite[Proposition 1.4.5]{meyer}, $(J(x_n))_n$ is a disjoint weakly $p$-summable sequence in $E^{\ast \ast}$, so it is a disjoint weakly null sequence in $E^{\ast \ast}$. Therefore, as $E^\ast$ has the disjoint $DPP_p$, we get from Theorem \ref{teorema1} that
    $ x_n^\ast(x_n) = J(x_n)(x_n^\ast)\longrightarrow 0.$ By applying Theorem \ref{teorema1} once again we conclude that $E$ has the disjoint $DPP_p$.
\end{proof}

\section{The disjoint Dunford-Pettis* property of order p}

In Banach spaces, the $DP^*P$ is defined by replacing Dunford-Pettis sets in the definition of the $DPP$ with limited sets. The following definition arises naturally by doing the same with the definition of the disjoint $DPP_p$ property introduced in \cite{ali}.

%\begin{itemize}
%%    \item Definir disjoint $DP^*P_p$: We say that $E$ has the disjoint $DP^*P_p$ if every disjoint weakly $p$-summable sequence in $E$ is limited.

%%    \item\textcolor{blue}{$E ~ DP^*P_p \implies E^* DP^*P_p$? E a recíproca? E bidual? E sublattice? Quocientes}

%%    \item Demonstrar uma caracterização semelhante ao \cite[Theorem 4.2]{ali}.

%%    \item Obter versões polinomiais das caracterizações obtidas no item anterior.
%\end{itemize}

 \begin{definition} \label{df1}\rm
     A Banach lattice $E$ has the {\it disjoint Dunford-Pettis* property of order $p$}, in short $E$ has the {\it disjoint $DP^*P_p$}, if every disjoint weakly $p$-summable sequence in $E$ is limited.
\end{definition}

\begin{examples} \label{obs2} \rm
 (1) It is easy to see that every Banach lattice $E$ with the $DP^\ast P$ or with the $SP_p^+$ has the disjoint $DP^\ast P_p$. \\
 (2) The Banach lattice $\ell_2$ has the $SP_1^+$ (see \cite[Example 3.9(ii)]{zeefou}), and consequently the disjoint $DP^\ast P_1$. However, $\ell_2$ fails the $DP^\ast P$ (see \cite[Remark 1.4]{cargalou}).\\
 (3) The Banach lattice $\ell_\infty$ has the $DP^\ast P$ (see \cite[Remark 3.2(c)]{borwein}), hence it has the disjoint $DP^\ast P_p$ for every $1 \leq p < \infty$. However, $\ell_\infty$ fails the $SP_p^+$ by Proposition \ref{hgty} because its norm is not order continuous.\\%, \textcolor{red}{because the natural inclusion $ \ell_{p^{\ast}} \hookrightarrow \ell_\infty$ is not compact \cite[p.\,879]{zeefou}}. \\
 (4) Since limited sets are Dunford-Pettis, every Banach lattice $E$ with the disjoint $DP^\ast P_p$ has the disjoint $DPP_p$. As to the converse, $c_0$ is a Banach lattice with the disjoint $DPP_p$ (because $c_0$  has the $DPP$), failing the disjoint $DP^\ast P_p$. Indeed, $(e_n)_n$ is a disjoint weakly $p$-summable sequence in $c_0$, but $\{ e_n : n \in \N \}$ is not a limited set. \\
 (5) At this point it is clear that disjoint $DP^\ast P_p$ is not inherited by sublattices: as we saw above, $c_0$ is a sublattice of $\ell_\infty$ that fails the disjoint $DP^\ast P_p$ while $\ell_\infty$ has the disjoint $DP^\ast P_p$. \\
 (6) The disjoint $DP^\ast P_p$ is inherited by complemented sublattices. Let $F$ be a complemented sublattice of the Banach lattice $E$ with the disjoint $DP^\ast P_p$.  If $(x_n)_n$ is a disjoint weakly $p$-summable sequence in $F$, then it is a disjoint weakly $p$-summable sequence in $E$, hence $\{ x_n : n \in \N \}$ is a limited subset of $E$. So, given a weak* null sequence $(y_j^\ast)_j$ in $F^\ast$, denoting by $P$ the projection from $E$ onto $F$, we have $y_j^\ast \circ P \cvfe 0$ in $E^\ast$, hence
$$ \sup_{n \in \N} |y_j^\ast(x_n)| = \sup_{n \in \N} y_j^\ast (P (x_n)) \longrightarrow 0   \text{~as } j \longrightarrow \infty,$$
which implies that $\{ x_n: n \in \N \}$ is a limited subset of $F$. Thus $F$ has the disjoint $DP^\ast P_p$.
   % \begin{enumerate}
    %    \item For every $1 \leq p < q < \infty$, if $E$ has the disjoint $DP^*P_q$, then $E$ has the disjoint $DP^*P_p$.

     %   \item If $E$  has the disjoint $DP^*P_p$, then $E$ has the disjoint $DPP_p$. The converse is not true though.

     %   \item If $E$ has the $DP^*P$, then $E$  has the disjoint $DP^*P_p$. {\normalfont Indeed, if $(x_n)_n \subset $ is a disjoint weakly $p$-summable sequence, then $\{x_n: n \in \N\}$ is, in particular, a relatively weakly compact subset of $E$, and hence $(x_n)_n$ is limited.}

     %%   \item If $E$ has the $SP_p^+$, then $E$ has the disjoint $DP^*P_p$. {\normalfont Indeed, if $(x_n)_n \subset $ is a disjoint weakly $p$-summable sequence, then $(|x_n|)$ is a positive disjoint weakly $p$-summable sequence (see \cite[Proposition 2.2]{zeefou}), and hence $\norma{x_n} = \norma{|x_n|} \to 0$. Thus $\{ x_n : n \in \N \}$ is relatively compact and consequently $(x_n)_n$ is limited.}
    %%\end{enumerate}
\end{examples}

%%\noindent \textbf{Question 1:} Example of a Banach lattice having the disjoint $DPP_p$ but failing to have the disjoint $DP^*P_p$. \textcolor{red}{$E = c_0$ has the $DPP$, so it has the disjoint $DPP_p$, but it fails to have the disjoint $DP^*P_p$.}

%\smallskip

%%\noindent \textbf{Question 2:} Example of a Banach lattice having the disjoint $DP^*P_p$ but failing to have the $DP^*P$. \textcolor{red}{Take $E$ to be a infinite dimensional reflexive Banach lattice with the $SP_p^+$.}

%%\smallskip

%%\smallskip
%%\noindent \textbf{Question 3:} From the remarks above we have that
%$$ SP_p^+ \cup DP^* \subset dDP^\ast P_p \subset dDPP_p. $$
%When do we have the equality of above inclusions?

From Example \ref{obs2}(4) it is a natural question to wonder which condition should be added to the disjoint $DPP_p$ to ensure the disjoint $DP^\ast P_p$. And from Example \ref{obs2}(1) it is also a  natural question to wonder which condition should be added to the disjoint $DPP_p$ to ensure the $SP_p^+$. % or whenever a Banach lattice with the disjoint $DP^\ast P_p$ has the $SP_p^+$}.
The following characterizations are needed to answer these questions.

\begin{theorem} \label{teorema2}
        The following are equivalent for a Banach lattice $E$. \\
    {\rm (a)} $E$ has the disjoint $DP^*P_p$.\\
    {\rm (b)} For any Banach space $Y$ such that $B_{Y^\ast}$ is sequentially weak* compact, every bounded linear operator $T\colon E \to Y$ is disjoint $p$-convergent. \\
    {\rm (c)} Every bounded linear operator $T \colon E \to c_0$ is disjoint $p$-convergent. \\
    {\rm (d)} For every disjoint weakly $p$-summable sequence $(x_n)_n$ in $E$ and every weak* null sequence $(x_n^\ast)_n$ $ E^\ast$, we have $x_n^\ast(x_n)\longrightarrow 0$.
\end{theorem}

\begin{proof}
    (a) $\Rightarrow$ (b) Assume by way of contradiction that there exist a Banach space $Y$ such that $B_{Y^\ast}$ is sequentially weak* compact and a bounded non disjoint $p$-convergent linear operator $T \colon E \to Y$. We can take a disjoint weakly $p$-summable sequence $(x_n)_n$ in $E$ such that $(T(x_n)_n)_n$ is not norm null in $Y$. Without loss of generality, we may assume that there exists $\varepsilon > 0$ such that $\norma{T(x_n)} \geq \varepsilon$ for every $n \in \N$. Thus, for each $n \in \N$ there exists $y_n^\ast \in S_{Y^\ast}$ such that $|y_n^\ast(T(x_n))| \geq \varepsilon$. The sequential weak* compactness of $B_{Y^\ast}$ gives a subsequence $(y_{n_j}^\ast)_j$ and $y^\ast \in Y^\ast$ such that $y_{n_j}^\ast \cvfe y^\ast$ in $Y^\ast$. It follows that   $T^\ast(y_{n_j}^\ast - y^\ast) \cvfe 0$ because $T^*$ is weak*-weak* continuous. We know that $(x_n)_n$ is a limited sequence in $E$ because $E$ has the disjoint $DP^\ast P_p$, so
    $$ |T^\ast(y_{n_j}^\ast - y^\ast)(x_{n_j})| \leq \sup_{k \in \N} |T^\ast(y_{n_j}^\ast - y^\ast)(x_{k})| \longrightarrow 0. $$
      But $(x_{n_j})_j$ is, in particular, a weakly null sequence in $E$, hence $T^\ast y^\ast (x_{n_j}) \longrightarrow 0$. Therefore,
    $$ \varepsilon \leq  |y_{n_j}^\ast(Tx_{n_j})| = |T^\ast y_{n_j}^\ast (x_{n_j})| \leq |T^\ast(y_{n_j}^\ast - y^\ast)x_{n_j}| + |T^\ast y^\ast (x_{n_j})| \longrightarrow 0, $$
    which is a contradiction.

    (b) $\Rightarrow$ (c) Just take $Y = c_0$ and use that duals of separable spaces have sequentially weak* compact unit balls.

    (c) $\Rightarrow$ (d) Let $(x_n)_n$ be a disjoint weakly $p$-summable sequence in $E$ and let $(x_n^\ast)_n$ be a weak* null sequence in $E^\ast$. By assumption, the operator $T \colon E \to c_0$ defined by $T(x) = (x_n^\ast(x))_n$ for every $x \in E$ is disjoint $p$-convergent. Therefore,
    $$|x_n^\ast(x_n)| \leq \sup_{k \in \N} |x_k^\ast(x_n)| = \norma{T(x_n)}_\infty \longrightarrow 0.$$

    (d) $\Rightarrow$ (a) Let $(x_n)_n$ be a disjoint weakly $p$-summable sequence in $E$. If $A := \{ x_n : n \in \N \}$ was not limited, there would exist $\varepsilon > 0$, a sequence $(x_{n_k})_n \subset A$ and a weak* null sequence $(x_k^\ast)_k$ in $E^\ast$ such that $|x_k^\ast (x_{n_k})| \geq \varepsilon$ for every $k \in \N$. This contradicts the fact that $E$ has the disjoint $DP^*P_p$ and completes the proof.
\end{proof}

%%%We recall that a Banach lattice $E$ has property (d) if for every disjoint weak* null sequence $(x_n^\ast)_n$ in  $E^\ast$ we have that $|x_n^\ast| \cvfe 0$. A bounded linear operator $T: E \to c_0$ is said to be a disjoint operator whenever its associated sequence is disjoint (see \cite{loumir2}).

We can go quite further for $\sigma$-Dedekind complete Banach lattices.

\begin{theorem} \label{teorema3}
The following are equivalent  for a $\sigma$-Dedekind complete  Banach lattice $E$.\\
    {\rm (a)} $E$ has the disjoint $DP^\ast P_p$.\\
    {\rm (b)} Every disjoint weakly $p$-summable sequence in $E$ is almost limited. \\
    {\rm (c)} For every disjoint weakly $p$-summable sequence $(x_n)_n$  in $ E$ and every disjoint weak* null sequence $(x_n^\ast)_n$ in $ E^\ast$, we have  $x_n^\ast(x_n)\longrightarrow 0$. \\
    {\rm (d)} For every disjoint weakly $p$-summable sequence $(x_n)_n$ in $ E$ and every positive disjoint weak* null sequence $(x_n^\ast)_n$ in $ E^\ast$, we have $x_n^\ast(x_n)\longrightarrow 0$.\\
    {\rm (e)} For every disjoint weakly $p$-summable sequence $(x_n)_n$ in $ E$ and every positive weak* null sequence $(x_n^\ast)_n$ in $ E^\ast$, we have  $x_n^\ast(x_n)\longrightarrow 0$.
\end{theorem}

\begin{proof} The implications (a) $\Rightarrow$ (b) and (c) $\Rightarrow$ (d) are  immediate. We skip the easy proofs of (a) $\Rightarrow$ (e) $\Rightarrow$ (d).

    (b) $\Rightarrow$ (c) Let $(x_n)_n$ be a disjoint weakly $p$-summable sequence in $E$ and let $(x_n^\ast)_n$ be a disjoint weak* null sequence in $E^\ast$. The set $ \{ x_k : k \in \N \}$ is almost limited by assumption, so
    $$ |x_n^\ast(x_n)| \leq \sup_{k \in \N} |x_n^\ast (x_k)| \longrightarrow 0. $$

    (d) $\Rightarrow$ (a) Let $(x_n)_n$ be a disjoint weakly $p$-summable sequence in $E$ and let $(x_n^\ast)_n$ be a weak* null sequence in $E^\ast$. As $(x_n)_n$ is disjoint, there exists a disjoint sequence $(y_n^\ast)_n $ in $E^\ast$ such that $|y_n^\ast| \leq |x_n^\ast|$ and $y_n^\ast(x_n)_n = x_n^\ast(x_n)_n$ for every $n \in \N$ (see \cite[Ex.\,22, p.\,77]{alip}). Since $E$ is $\sigma$-Dedekind complete, $x_n^\ast \cvfe 0$ in $E^\ast$ and $(y_n^\ast)_n \subset E^\ast$ is a disjoint sequence with $|y_n^\ast| \leq |x_n^\ast|$ for every $n \in \N$, we get from \cite[Lemma 2.2]{chen} that $|y_n^\ast| \cvfe 0$ in $E^\ast$. And since $(x_n)_n$ is disjoint and weakly $p$-summable, $(|x_n|)_n$ is also a disjoint weakly $p$-summable sequence in $E$ (see \cite[Proposition 2.2]{zeefou}). So, $|y_n^\ast| (|x_n|) \longrightarrow 0$ by assumption, therefore
    $$ |x_n^\ast(x_n)| = |y_n^\ast(x_n)| \leq |y_n^\ast|(|x_n|) \longrightarrow 0,  $$
   proving that $E$ has the disjoint $DP^\ast P_p$.
%    It is easy to see that (a) $\Rightarrow$ (e) $\Rightarrow$ (d) and we are done.
\end{proof}

It is well known that if a Banach space with the Grothendieck property has the $DPP$, then it has the $DP^\ast P$. This holds because Dunford-Pettis sets and  limited sets coincide in Banach spaces with the Grothendieck property. The lattice counterpart of the Grothendieck property is the so-called weak Grothendieck property introduced in \cite{mach}: a Banach lattice $E$  has the weak Grothendieck property if every disjoint weak* null sequence in $E^\ast$ is weakly null. It is under this property that the disjoint $DPP_p$ and the disjoint $DP^*P_p$ coincide:

%It was observed in the Introduction of \cite{galmir} that in a Banach lattice with the weak Grothendieck property, the classes of almost Dunford-Pettis and almost limited sets coincide. From this fact and Theorem \ref{teorema3}, we have the following:

\begin{corollary} \label{cor1}
        A $\sigma$-Dedekind complete Banach lattice has the disjoint $DP^\ast P_p$ if and only if it has the disjoint $DPP_p$ and the weak Grothendieck property.
\end{corollary}

\begin{proof} Let $E$ be a $\sigma$-Dedekind complete Banach lattice with the disjoint $DPP_p$ and the weak Grothendieck property. Given a disjoint weakly $p$-summable sequence $(x_n)_n$  in $E$, the disjoint $DPP_p$ gives that $\{ x_n : n \in \N\}$ is an almost Dunford-Pettis subset of $E$. In Banach lattices with the weak Grothendieck property, the classes of almost Dunford-Pettis and almost limited sets coincide (see \cite[p.\,2]{galmir}), so $\{ x_n : n \in \N\}$ is an almost limited subset of $E$.  By Theorem \ref{teorema3} we conclude that $E$ has the disjoint $DP^\ast P_p$. The converse implication was settled in Example \ref{obs2}(4).
%%    Assume that $E$ has the disjoint $DPP_p$ and let $(x_n)_n$ be a disjoint weakly $p$-summable sequence. Since $E$ has the disjoint $DPP_p$, $\{ x_n : n \in \N \}$ is a Dunford-Pettis subset of $E$. In particular, $\{ x_n : n \in \N \}$ is an almost Dunford-Pettis set, and then it is almost limited.  By Theorem \ref{teorema3} we conclude that $E$ has the disjoint $DP^\ast P_p$.
\end{proof}

It was observed in Example \ref{obs2}(4) that $c_0$ is a Banach lattice with the disjoint $DPP_p$ failing the disjoint $DP^\ast P_p$. This does not contradict Corollary \ref{cor1} because $c_0$ is a $\sigma$-Dedekind complete Banach lattice without the weak Grothendieck property. It is worth noticing that there exist $\sigma$-Dedekind complete Banach lattices with the weak Grothendieck property failing both the disjoint $DPP_p$ and the disjoint $DP^\ast P_p$. Take, for instance, $E = \ell_2$ and $p = 2$.

%%\noindent \textbf{Question 4:} Example of a $\sigma$-Dedekind complete Banach lattice with the weak Grothendieck property but failing the disjoint $DP^\ast P_p$. \textcolor{red}{$E = \ell_2$ does not have the disjoint $DP^\ast P_2$. Indeed, the cannonical sequence $(e_n)_n$ is weakly $2$-summable in $\ell_2$, but $\{ e_n : n\in \N \}$ is not a limited subset of $\ell_2$.}

%\medskip

Inspired by \cite[Proposition 3.3]{chen} we show that it is under an order continuous norm that the disjoint $DP^\ast P_p$ coincide with the $SP_p^+$:

\begin{proposition}\label{hgty}
    A Banach lattice $E$ has the $SP_p^+$ if and only if $E$ has order continuous norm and the disjoint $DP^\ast P_p$.
\end{proposition}

\begin{proof}
    Assume first that $E$ has the $SP_p^+$. In Example \ref{obs2}(1) we saw that $E$ has the disjoint $DP^\ast P_p$. In order to prove that $E$ has order continuous norm, let $(x_n)_n$ be a disjoint order bounded sequence in $E^+$. So, $(x_n)_n$ is weakly $p$-summable (see the proof of \cite[Lemma 2.2]{botgarmir}), hence  $\norma{x_n} \longrightarrow 0$ by assumption. Therefore $E$ has order continuous norm by \cite[Theorem 2.4.2]{meyer}.

     Suppose now that $E$ has order continuous norm and the disjoint $DP^\ast P_p$. If $E$ does not have the $SP_p^+$, there exists a positive disjoint weakly $p$-summable sequence $(x_n)_n \subset E$ such that $\norma{x_n} = 1$ for all $n \in \N$. For each $n \in \N$, there exists $x_n^\ast \in E^\ast$ with $\norma{x_n^\ast} = 1$ such that $x_n^\ast (x_n) = 1$. Since $(x_n)_n$ is a disjoint sequence, there exists a disjoint sequence $(y_n^\ast)_n \subset E^\ast$ such that $|y_n^\ast| \leq |x_n^\ast|$ and $y_n^\ast(x_n) = x_n^\ast(x_n) = 1$ for every $n \in \N$ \cite[Ex.\,22, p.\,77]{alip}. As $E$ has order continuous norm, by \cite[Corollary 2.4.3]{meyer} we have $y_n^\ast \cvfe 0$ in $E^\ast $, which contradicts the fact that $E$ has the disjoint $DP^\ast P_p$.
\end{proof}

%Next we shall prove that, for $p \geq 2$,  in $\sigma$-Dedekind complete Banach lattices the disjoint $DP^*P$ coincide with the $p$-weak $DP^*P$ property introduced
According to \cite{ardakani}, a Banach lattice $E$ has the $p$-weak $DP^\ast$ property if every weakly $p$-compact subset of $E$ is almost limited; or, equivalently, for every weakly $p$-summable sequence $(x_n)_n$ in $ E$ and every disjoint weak* null sequence $(x_n^\ast)_n$ in $ E^\ast$, it holds $x_n^\ast (x_n)\longrightarrow 0$. Bearing in mind the equivalence (a) $\Leftrightarrow$ (c) in Theorem \ref{teorema3}, \cite[Theorem 3.8]{ardakani} establishes that a $\sigma$-Dedekind complete Banach lattice $E$ with type $1 < q \leq 2$ has the $p$-weak $DP^\ast$ property, $p \geq q^\ast$, if and only if $E$ %for every disjoint weakly $p$-summable sequence $(x_n)_n \subset E$ and every disjoint weak* null sequence $(x_n^\ast)_n \subset E^\ast$, we have that $x_n^\ast (x_n)\longrightarrow 0$. Observe that in $\sigma$-Dedekind complete Banach lattices, this last condition is equivalent to say that $E$
has the disjoint $DP^\ast P_p$. Next we show that, for $p \geq 2$, the nontrivial type of $E$ can be dropped. In summary, the following result improves upon \cite[Theorem 3.8]{ardakani}.  %To prove this characterization, the authors in \cite{ardakani} used Fourie-Zeekoei's result \cite[Lemma 3.4]{fouzee}. Nevertheless, by applying \cite[Theorem 3.6]{galmir2}, we have that the disjoint $DP^\ast P_p$ and the $p$-weak $DP^\ast$ property coincide in every $\sigma$-Dedekind complete Banach lattice for $p \geq 2$.

\begin{theorem} \label{teorema4} Let $E$ be a $\sigma$-Dedekind complete Banach lattice.\\
 {\rm (a)} If $E$ has the $p$-weak $DP^\ast$ property then $E$  has the disjoint $DP^\ast P_p$.\\
{\rm (b)} For $p \geq 2$, $E$ has the $p$-weak $DP^\ast$ property if and only if $E$  has the disjoint $DP^\ast P_p$.
\end{theorem}

\begin{proof} (a)  If $(x_n)_n$ is a disjoint weakly $p$-summable sequence, then $\{ x_n : n \in \N \}$ is a weakly $p$-compact subset of $E$, hence it almost limited because $E$ has the $p$-weak $DP^\ast$ property. By Theorem \ref{teorema3} we conclude that $E$ has the disjoint $DP^\ast P_p$.

\noindent(b) Assume that $E$ has the disjoint $DP^\ast P_p$ and fails the $p$-weak $DP^\ast$ property. In this case there exists a weakly $p$-summable sequence $(x_n)_n$ in $E$ and a disjoint weak* null sequence $(x_n^\ast)_n$ in $E^\ast$ such that $(x_n^\ast(x_n))_n$ does not converge to zero. Without loss of generality, we may assume that $|x_n^\ast|(|x_n|) \geq |x_n^\ast(x_n)| \geq \varepsilon$ for every $n \in \N$ and some $\varepsilon > 0$. The $\sigma$-Dedekind completeness of $E$ gives $|x_n^\ast| \cvfe 0$ in $E^\ast$ by \cite[Proposition 1.4]{wnukdual}.  Put $n_1 = 1$. As $|x_n^\ast|(4|x_{n_1}|) \longrightarrow 0$, there exists $n_2 > n_1$ such that $|x_{n_2}^\ast| (4 |x_{n_1}|) < 1/2$. Again, since $|x_n^\ast|\left(4^2 \sum\limits_{j=1}^2 |x_{n_j}|\right) \longrightarrow 0$, there exists $n_3 > n_2$ such that $|x_{n_3}^\ast| \left(4^2 \sum\limits_{j=1}^2 |x_{n_j}|\right) < 1/2^2$.  Inductively we construct a strictly increasing sequence $(n_k)_k \subset \N$ such that
    $$ |x_{n_{k+1}}^\ast|\left(4^k \displaystyle\sum_{j=1}^k |x_{n_j}|\right) < \frac{1}{2^k}  \text{~for every } k \in \N. $$
 For any $k \in \N$, setting $x = \sum\limits_{k=1}^\infty 2^{-k} |x_{n_k}|$ and
    $$ u_k = \left ( |x_{n_{k+1}}| -4^k \sum_{j=1}^k |x_{n_j}| -2^{-k} x \right )^+, $$
    we obtain from \cite[Lemma 4.35]{alip} that $(u_k)_k$ is a positive disjoint sequence such that $0 \leq u_k \leq |x_{n_{k+1}}|$ for every $k \in \N$. Since $(x_{n_{k+1}})_k$ is weakly $p$-summable and $p \geq 2$, we get from \cite[Theorem 3.6]{galmir2} that $(u_k)_k$ is weakly $p$-summable, hence $x_{n_{k+1}}^\ast (u_k) \longrightarrow 0$ because $E$ has the disjoint $DP^\ast P_p$. From
    \begin{align*}
        |x_{n_{k+1}}^\ast|(u_k) & \geq |x_{n_{k+1}}^\ast| (|x_{n_{k+1}}|) - |x_{n_{k+1}}^\ast| \left( 4^k \sum_{j=1}^k |x_{n_j}| \right) - 2^{-k} |x_{n_{k+1}}^\ast| (x) \\
            & \geq \varepsilon - \frac{1}{2^k} - 2^{-k} |x_{n_{k+1}}^\ast| (x)
    \end{align*}
for every $k$, we conclude that $\liminf\limits_{n \to \infty} |x_{n_{k+1}}^\ast| (u_k) \geq \varepsilon$, a contradiction that completes the proof.
\end{proof}

%Observe that we only use the condition $p \geq 2$ in the proof of Theorem \ref{teorema4} to prove that if $E$ has the disjoint $DP^\ast P_p$, then $E$ has the $p$-weak $DP^\ast$. In fact, the proof given above shows that every $\sigma$-Dedekind complete Banach lattice with the $p$-weak $DP^\ast$ property for some $1 \leq p < \infty$ has the disjoint $DP^\ast P_p$.
%We have already seen that the converse holds for $p \geq 2$.

In the recent paper \cite{ardakani2}, H. Ardakani and K. Amjadi introduced the positive $DP^\ast$ property of order $p$ considering the positively limited sets introduced in \cite{ardachen}. It was observed before \cite[Theorem 3.6]{ardakani2} that a Banach lattice $E$ with the property (d) (see \cite{loumir2} for details) has the positive $DP^\ast$ property of order $p$ if and only if $E$ has the $p$-weak $DP^\ast$. Since $\sigma$-Dedekind complete Banach lattices have the property (d), the following is a straightforward consequence of Theorem \ref{teorema4}:

\begin{corollary} For $p \geq 2$, a $\sigma$-Dedekind complete Banach lattice has the positive $DP^\ast$ property of order $p$ if and only if it has the disjoint $DP^\ast P_p$.
\end{corollary}

Combining the corollary above with Theorem \ref{teorema3} we get that, for $p \geq 2$, \cite[Theorem 3.6]{ardakani2} holds for every $\sigma$-Dedekind complete Banach lattice.

\begin{remark}\rm For $1 \leq p < \infty$, it is worth mentioning that  every $\sigma$-Dedekind complete Banach lattice with the disjoint $DP^\ast P_p$ which is weak $p$-consistent (see \cite[Definition 2.4]{zeefou}) has the $p$-weak $DP^\ast$ property. %\textcolor{blue}{(Isso \'e \'obvio?)}
This follows by the same argument used in the proof of Theorem \ref{teorema4}(b) by using the (definition of the) weak $p$-consistency of $E$ instead of \cite[Theorem 3.6]{galmir2}.
\end{remark}

\section{Polynomial properties}

Inspired by the results obtained for polynomial versions of the $wDPP$ and of the $wDP^\ast P$ in \cite{wangshibu}, especially \cite[Theorems 3.8 and 4.3]{wangshibu}, we investigate in this section the behavior of the disjoint $DPP_p$ and of the disjoint $DP^\ast P_p$ when linear operators and linear functionals are replaced with homogeneous polynomials. % inspired in the already known results for the $wDPP$ and the $wDP^\ast P$ (see, respectively, Theorem 3.8 and Theorem 4.3 in \cite{wangshibu}).
In particular, polynomial versions of our Theorems \ref{teorema1}, \ref{teorema2} and \ref{teorema3} shall be proved. %replacing bounded linear operators and linear functionals by homogeneous polynomials.

We begin by recalling the terminology and a few properties concerning regular polynomials and the positive projective symmetric tensor product. Details can be found in \cite{bubuskes} and \cite{loane}.

A $k$-homogeneous polynomial between Riesz spaces $P \colon E \to F$ is positive  if its associated symmetric multilinear operator $T_P\colon E \times \overset{k}{\cdots} \times E \to F$ is positive, meaning that $T_P(x_1, \ldots, x_k) \geq 0$ for all $x_j \in E_j^+, j = 1,\ldots, k$. The difference of two positive $k$-homogeneous polynomials is called a regular homogeneous polynomial, and the set of all these polynomials is denoted by $\mathcal{P}^r(^kE, F)$. When $F$ is the scalar field we simply write $\mathcal{P}^r(^kE)$. If $E$ and $F$ are Banach lattices with $F$ Dedekind complete, then $\mathcal{P}^r(^kE, F)$ is a Banach lattice with the regular norm $\norma{P}_r = \norma{|P|}$, where $|P|$ denotes the absolute value of the regular polynomial $P \colon E \to F$.

For a Banach lattice $E$, we denote the
$k$-fold positive projective symmetric tensor product of $E$ by $\widehat{\otimes}_{s, |\pi|}^k E$, which is a Banach lattice
endowed with the positive projective symmetric tensor norm $\norma{\cdot}_{s, |\pi|}$. As usual, we write $\otimes^k x = x \, \otimes \stackrel{k}{\cdots} \otimes \, x$ for every $x \in E$.  %Let  $\theta_k \colon E \to \widehat{\otimes}_{s, |\pi|}^k E$ be the canonical $k$-homogeneous polynomial given by $\theta_k(x) = \otimes^k x$.
For every $P \in \mathcal{P}^r(^kE, F)$ there exists a unique regular linear operator $P^\otimes \colon \widehat{\otimes}_{s, |\pi|}^k E \to  F$, called the linearization of $P$, such that $P(x) = P^\otimes(\otimes^kx)$ for every $x \in E$. Moreover, the correspondence
\begin{equation}P \in \mathcal{P}^r(^kE, F) \mapsto P^\otimes \in \mathcal{L}^r\left(\widehat{\otimes}_{s, |\pi|}^k E; F\right)  \label{omku}\end{equation}
% The opereator
%$$\Phi\colon \mathcal{L}^r\left(\widehat{\otimes}_{s, |\pi|}^k E; F\right) \to \mathcal{P}^r(^k E; F)~,~ \Phi(T) = T \circ \theta_k,% \quad T \in \mathcal{L}^r(\widehat{\otimes}_{s, |\pi|}^k E; F),
%$$
is an isometric isomorphism and a lattice homomorphism. %For each $P \in \mathcal{P}^r(^k E; F)$, the regular linear operator $\Phi^{-1} (P)\colon \widehat{\otimes}_{s, |\pi|}^k E \to F $ is called the linearization of $P$, and it is denoted by $P^\otimes$.

To prove the first theorem of the section we need two lemmas. The first one can be found within the proof of \cite[Theorem 3.8]{wangshibu}.

\begin{lemma} \label{lemapol}
    Let $(x_n^\ast)_n$ be a positive weakly null sequence in $E^\ast$ and $k \in \N$.    For each $n \in \N$, define $P_n(x) = (x_n^\ast (x))^k$ for every $x \in E$, that is, $P_n = (x_n^*)^k \in {\cal P}^r(^kE)$. Then $P_n^\otimes = \otimes ^k x_n^\ast $ for every $n \in \N$ and $(P_n)_n$ is a positive weakly null sequence in $\left ( \widehat{\otimes}_{s, |\pi|}^k E\right )^\ast$. %% Moreover, if $(x_n^\ast)_n$ is a disjoint sequence, then $(P_n^\otimes)_n$ is a disjoint sequence in $\left ( \widehat{\bigotimes}_{k, s, \pi} E\right )^\ast$.
\end{lemma}

The proof of the next lemma is essentially contained in the proof of \cite[Proposition 2.34]{dineen}. We give a short reasoning for the convenience of the reader.

\begin{lemma} \label{lemapol2}
If $(x_n)_n$ is a weakly null Dunford-Pettis sequence in a Banach space $X$, then $Q(x_n)\longrightarrow 0$ for every homogeneous polynomial $Q\colon X \to \K$, where $\K = \R$ or $\mathbb{C}$.
\end{lemma}

\begin{proof}
    The fact that $(x_n)_n$ is a weakly null sequence gives that the result holds for linear functionals, that is, for $1$-homogeneous polynomials. Suppose that the result holds for $k$-homogeneous polynomials and let us prove that it holds for $(k+1)$-homogeneous polynomials. Let $Q \colon X \to \K$ be a $(k+1)$-homogeneous polynomial. The map
    $$P\colon X \to X^\ast ~, ~[P(x)](z) = T_Q(z,x, \overset{k}{\cdots}, x),$$
       is a $k$-homogeneous polynomial, so $x^{\ast \ast} \circ P \colon X \to \K$ is a $k$-homogeneous polynomial for every $x^{\ast \ast} \in X^{\ast \ast}$. By the induction hypothesis, $x^{\ast \ast} (P(x_n)) \longrightarrow 0$ for every $x^{\ast \ast} \in X^{\ast \ast}$, which implies that $(P(x_n))_n$ is a weakly null sequence in $X^\ast$. As $(x_n)_n$ is a Dunford-Pettis sequence, it follows that $Q(x_n) = (P(x_n))(x_n)\longrightarrow 0$ and we are done.
\end{proof}

\begin{theorem} \label{teopol1}
    For a Banach lattice $E$ and a positive integer $k$, the following statements are equivalent.  \\
{\rm (a)} $E$ has the disjoint $DPPp$. \\
{\rm (b)} For any Banach lattice $F$ with order continuous norm, every positive weakly compact $k$-homogeneous polynomial $P\colon E \to F$ is disjoint $p$-convergent.\\%, that is it takes disjoint weakly $p$-summable sequences to norm null sequences. \\
{\rm (c)} Every positive weakly compact $k$-homogeneous polynomial $P\colon E \to c_0$ is disjoint $p$-convergent. \\
{\rm (d)} For every disjoint weakly $p$-summable sequence $(x_n)_n$ in $E$ and for every positive weakly null sequence $(P_n)_n$ in $\mathcal{P}^r(^k E)$, it holds $P_n(x_n)\longrightarrow 0$. \\
{\rm (e)}  For every disjoint weakly $p$-summable sequence $(x_n)_n$ in $ E$ and for every positive disjoint weakly null sequence $(P_n)_n$ in $ \mathcal{P}^r(^k E)$, it holds $P_n(x_n)\longrightarrow 0$. \\
{\rm (f)} For every disjoint weakly $p$-summable sequence $(x_n)_n$ in $ E$ and for every disjoint weakly null sequence $(P_n)_n$ in $ \mathcal{P}^r(^k E)$, it holds $P_n(x_n)\longrightarrow 0$.

\end{theorem}

\begin{proof} The implications (b) $\Rightarrow$ (c), (d) $\Rightarrow$ (e) and (f) $\Rightarrow$ (e) are immediate.

    (a) $\Rightarrow$ (b) Supposing by way of contradiction that (b) fails, there exist a Banach lattice $F$ with order continuous norm, a positive weakly compact $k$-homogeneous polynomial $P \colon E \to F$ and a disjoint weakly $p$-summable sequence $(x_n)_n$ in $E$ such that $\lim\limits_{n\to \infty} P(x_n) \neq 0$. Without loss of generality, we may assume that $\norma{P(x_n)} \geq \varepsilon$ for every $n \in \N$ and some  $\varepsilon>0$. So, for each $n \in \N$, there exists $y_n^\ast \in S_{F^\ast}$ such that $|y_n^\ast(P(x_n))| = \norma{P(x_n)} \geq \varepsilon.$ Since $P\colon E \to F$ is a positive weakly compact polynomial and $F$ has order continuous norm, we obtain from \cite[Theorem 4.1]{libu} that $P^\otimes\colon \widehat{\otimes}_{s, |\pi|}^k E \to F$ is a weakly compact linear operator, hence its adjoint $(P^\otimes)^\ast \colon F^\ast \to \left ( \widehat{\otimes}_{s, |\pi|}^k E \right )^\ast$ is weakly compact as well. Therefore, there exists a subsequence $(y_{n_j}^\ast)$ of $(y_n^\ast)_n$ such that $((P^\otimes)^\ast(y_{n_j}^\ast))_j$ is weakly convergent in $\left ( \widehat{\otimes}_{s, |\pi|}^k E \right )^\ast$. By using the identification (\ref{omku}) between $\left ( \widehat{\otimes}_{s, |\pi|}^k E \right )^\ast$ and $\mathcal{P}^r(^k E)$, there exists $Q \in \mathcal{P}^r(^k E)$ such that
    $$(y_{n_j}^\ast \circ P)^\otimes = y_{n_j}^\ast \circ P^\otimes = (P^\otimes)^\ast(y_{n_j}^\ast) \cvf Q^\otimes.$$
Applying the identification in the other direction we get that $(y_{n_j}^\ast \circ P - Q)_j$ is a weakly null sequence in $\mathcal{P}^r(^k E)$.

Using the disjoint $DPP_p$ of $E$, the set $\{ x_n : n \in \N \}$ is almost Dunford-Pettis, consequently $A := \sol{\{x_n : n \in \N \}}$ is a solid almost Dunford-Pettis subset of $E$ (see the proof of the theorem in \cite[p.\,112]{chenli}). By \cite[Theorem 3.4]{shiwangbu} we have $\sup\limits_{z \in A} |Q_j(z)| \longrightarrow 0$ for every disjoint weakly null sequence $(Q_j)_j$ in $\mathcal{P}^r(^k E)$. This implies that $Q_j(z_j) \longrightarrow 0$ for every disjoint sequence $(z_j)_j$ in $ A$ and every disjoint weakly null sequence $(Q_j)_j$ in $ \mathcal{P}^r(^k E)$. From \cite[Lemma 3.2]{shiwangbu} it follows that $Q_j(z_j) \longrightarrow 0$ for every disjoint sequence $(z_j)_j$ in $ A$ and every weakly null sequence $(Q_j)_j$ in $ \mathcal{P}^r(^k E)$. Since $(x_{n_j})_j$ is a disjoint sequence contained in $A$, we have $\lim\limits_{j \to \infty} (y_{n_j}^\ast \circ P - Q)(x_{n_j}) = 0$. Finally, using that $(x_{n_j})_{j}$ is, in particular, a weakly null Dunford-Pettis sequence in $E$, we get from Lemma \ref{lemapol2} that
    $Q(x_{n_j}) \longrightarrow 0$. Therefore,
    $$ \varepsilon \leq |y_{n_j}^\ast (P(x_{n_j}))| \leq |(y_{n_j}^\ast \circ P - Q)(x_{n_j})| + |Q(x_{n_j})| \longrightarrow 0,  $$
a contradiction that completes the proof of (a) $\Rightarrow$ (b).

%  : Take $F = c_0$.

   %%% (c) $\Rightarrow$ (a): Let $(x_n)_n \subset E$ be a disjoint weakly $p$-summable sequence and let $(x_n^\ast)_n$ be positive weakly null sequence in $E$. For each $n \in \N$, letting $P_n(x) = (x_n^\ast (x))^k$ for every $x \in E$,    we have by following the same argument used in the proof of \cite[Theorem 3.1]{bombal} that $P_n^\otimes = x_n^\ast \otimes \overset{k}{\cdots} \otimes x_n^\ast$, $n \in \N$, defines a weakly null sequence in $\left ( \widehat{\bigotimes}_{k, s, \pi} E\right )^\ast$. Consequently, $T(x) = (P_i^\otimes (x))_i$ defines a weakly compact operator from $widehat{\bigotimes}_{k, s, \pi} E$ to $c_0$, and so the $k$-homogeneous polynomial $P: E \to c_0$ such that $P^\otimes = T$ is weakly compact. Moreover, since $(x_n^\otimes)_n$ is a positive sequence, $P$ is a positive polynomial, and by the assumption $\lim_{n \to \infty} P(x_n)_n = 0$. Therefore    $$ |P_n(x_n)_n| = |P_n^\otimes (\theta_k(x_n)_n)| \leq \sup_{i \in \N} |P_i^\otimes (\theta_k(x_n)_n)| = \norma{T(\theta_k(x_n)_n)}= \norma{P(x_n)_n} \to 0. $$

 (c) $\Rightarrow$ (d) Let $(x_n)_n$ be a disjoint weakly $p$-summable sequence  in $E$ and let $(P_n)_n$ be a positive weakly null sequence in  $\mathcal{P}^r(^k E)$. Considering the identification (\ref{omku}) again % between $\mathcal{P}^r(^k E)$ and $\left ( \widehat{\otimes}_{s, |\pi|}^k E \right )^\ast$,
 we have that $(P_n^\otimes)_n$ is a positive weakly null sequence in $\left ( \widehat{\otimes}_{s, |\pi|}^k E \right )^\ast$. So,
    $T(z) = (P_i^\otimes (z))_i$ defines a positive weakly compact operator from $\widehat{\otimes}_{s, |\pi|}^k E $ to $c_0$. Hence the  positive $k$-homogeneous polynomial $P\colon E \to c_0$ such that $P^\otimes = T$ is a positive weakly compact polynomial by \cite[Theorem 4.1]{libu}. The assumption gives $\lim\limits_{n \to \infty} P(x_n) = 0$, from which we get
    $$ |P_n(x_n)| = |P_n^\otimes (\otimes^k x_n)| \leq \sup_{i \in \N} |P_i^\otimes (\otimes^k x_n)| = \norma{T(\otimes^k x_n)}= \norma{P(x_n)} \longrightarrow 0. $$

 %   (d) $\Rightarrow$ (e): Immediate.

    (e) $\Rightarrow$ (f) Let $(x_n)_n$ be a disjoint weakly $p$-summable sequence in $E$ and let $(P_n)_n$ be a  disjoint weakly null sequence in $\mathcal{P}^r(^k E)$. By \cite[Proposition 2.2]{zeefou} we obtain that $(|x_n|)_n$ is a disjoint weakly $p$-summable sequence in $E$ and by  \cite[Proposition 1.3]{wnukdual} we know that $(|P_n|)_n$ is a positive disjoint weakly null sequence in $\mathcal{P}^r(^k E)$. The assumption gives $|P_n|(|x_n|)$, from which it follows that $|P_n(x_n)| \leq |P_n|(|x_n|) \longrightarrow 0.$

  %  (f) $\Rightarrow$ (e): Immediate.

   (e) $\Rightarrow$ (a) Let $(x_n)_n$ be a disjoint weakly $p$-summable sequence in $E$ and let $(x_n^\ast)_n$ be positive  weakly null sequence in $E^*$. For each $n \in \N$, letting $P_n(x) = (x_n^\ast (x))^k$ for every $x \in E$,
    we have by Lemma \ref{lemapol} that $(P_n^\otimes)_n = (\otimes^k x_n^\ast)_n$ is a positive weakly null sequence in $\left ( \widehat{\otimes}_{s, |\pi|}^k E\right )^\ast$. Using once again that the correspondence (\ref{omku}) is an isometric isomorphism and a lattice homomorphism, $(P_n)_n$ is a positive weakly null sequence in $\mathcal{P}^r(^k E)$. By \cite[Lemma 3.2]{wangshibu} there exists a disjoint weakly null sequence $(Q_n)_n$ in $ \mathcal{P}^r(^k E)$ such that $Q_n(x_n) = P_n(x_n)$ for every $n \in \N$. The assumption gives
    $$ (x_n^\ast (x_n))^k = P_n(x_n) = Q_n(x_n)\longrightarrow 0, $$
therefore $E$ has the disjoint $DPP_p$ by Theorem \ref{teorema1}.
%%    So far we proved that (a), (b), (c), (d) and (e) are equivalent.   Since (f) $\Rightarrow$ (e) is immediate, it suffices us to check that (e) $\Rightarrow$ (f). So, let $(x_n)_n \subset E$ be a disjoint weakly $p$-summable sequence and let $(P_n) \subset \mathcal{P}^r(^k E)$ be a disjoint weakly null sequence. By \cite[Proposition 2.2]{zeefou} and \cite[Proposition 1.3]{wnukdual}, we obtain, respectively, that $(|x_n|)_n$ is a disjoint weakly $p$-summable sequence in $E$ and that $(|P_n|)_n$ is a positive disjoint weakly null sequence in $\mathcal{P}^r(^k E)$. So, by the assumption $|P_n|(|x_n|)$, and consequently $|P_n(x_n)_n| \leq |P_n|(|x_n|) \to 0.$
\end{proof}

In order to establish our last result, recall that a sequence  $(P_n)_n$ of scalar-valued regular $k$-homogeneous polynomials on $E$ is said to be weak* null if $(P_n^\otimes)_n$ is a weak* null sequence in $\left ( \widehat{\otimes}_{s, |\pi|}^k E\right )^\ast$, or, equivalently, if $P_n(x) \longrightarrow 0$ for every $x \in E$ (see \cite[Lemma 4.1]{wangshibu}).

\begin{theorem} \label{teopol2}
            For a $\sigma$-Dedekind complete Banach lattice $E$ and a positive integer $k$, the following statements are equivalent. \\
    {\rm (a)} $E$ has the disjoint $DP^*P_p$.\\
    {\rm (b)} For any Banach lattice $F$ such that $B_{F^\ast}$ is sequentially weak* compact, every positive $k$-homogeneous polynomial $P\colon E \to F$ is disjoint $p$-convergent. \\
    {\rm (c)} Every positive $k$-homogeneous polynomial $P \colon E \to c_0$ is disjoint $p$-convergent. \\
    {\rm (d)} Every regular $k$-homogeneous polynomial $P \colon E \to c_0$ is disjoint $p$-convergent. \\
    {\rm (e)} For every disjoint weakly $p$-summable sequence $(x_n)_n$ in $E$ and every weak* null sequence $(P_n)_n$ in $\mathcal{P}^r(^k E)$, it holds $P_n(x_n)\longrightarrow 0$. \\
    {\rm (f)} For every disjoint weakly $p$-summable sequence $(x_n)_n$ in $E$ and every disjoint weak* null sequence $(P_n)_n$ in $ \mathcal{P}^r(^k E)$, it holds $P_n(x_n)\longrightarrow 0$. \\
    {\rm (g)} For every disjoint weakly $p$-summable sequence $(x_n)_n$ in $E$ and every positive disjoint weak* null sequence $(P_n)_n$ in $\mathcal{P}^r(^k E)$, it holds $P_n(x_n)\longrightarrow 0$. \\
    {\rm (h)} For every disjoint weakly $p$-summable sequence $(x_n)_n$ in $ E$ and every positive weak* null sequence $(P_n)_n$ in $ \mathcal{P}^r(^k E)$, it holds $P_n(x_n)\longrightarrow 0$.
\end{theorem}

\begin{proof}  The implications (b) $\Rightarrow$ (c), (e) $\Rightarrow$ (f), (f) $\Rightarrow$ (g), (e) $\Rightarrow$ (h) and (h) $\Rightarrow$ (g) are immediate.

    (a) $\Rightarrow$ (b) The case $k=1$ follows from Theorem \ref{teorema2}. Assume that (b) holds for $k > 1$ and le us prove that it holds for $k+1$. To do so, let $P\colon E \to F$ be a positive $(k+1)$-homogeneous polynomial and let $(x_n)_n$ be a disjoint weakly $p$-summable sequence in $E$. For each $x \in E$, define $Q_x\colon E \to F$ by $Q_x(y) = T_P(|x|, y, \overset{k}{\dots}, y)$, where $T_P \colon E^{k+1}\to F$ denotes the positive symmetric $(k+1)$-linear operator associated to $P$. Thus, $Q_x$ is a positive $k$-homogeneous polynomial. By \cite[Proposition 2.2]{zeefou}, $(|x_n|)_n$ is a disjoint weakly $p$-summable sequence, so the assumption gives $\lim\limits_{n \to \infty} Q_x(|x_n|) = 0$. In particular, using that $|T_P(x, x_n, \overset{k}{\cdots}, x_n)| \leq T_P(|x|, |x_n|, \overset{k}{\cdots}, |x_n|)$, we have that
    $$ \norma{T_P(x, x_n, \overset{k}{\cdots}, x_n)} \leq \norma{T_P(|x|, |x_n|, \overset{k}{\cdots}, |x_n|)} = \norma{Q_x(|x_n|)} \longrightarrow 0 $$
    for every $x \in E$. Proceeding as in the proof of \cite[Proposition 3.1]{cargalou}, we can choose a sequence $(\varphi_n)_n$ in $S_{F^\ast}$ such that $\varphi_n(P(x_n)) = \norma{P(x_n)}$ for every $n \in \N$. So,
     $$T\colon E \to c_0~,~T(x) = (\varphi_i(T_{P}(x, x_i, \overset{k}{\cdots}, x_i)))_i,$$
     is a well defined bounded linear operator. The disjoint $DP^\ast P_p$ of $E$, together with Theorem \ref{teorema2}, yields that $T$ is disjoint $p$-convergent operator. This implies that $\lim\limits_{n \to \infty} T(x_n)_n = 0$. For every $n \in \N$,
    $$ \norma{P(x_n)} = \varphi_n(P(x_n)) = \varphi_n(T_P(x_n, x_n, \overset{k}{\cdots}, x_n)) = \sup_{i \in \N} |\varphi_i(T_{P}(x_n, x_i, \overset{k}{\cdots}, x_i))| = \norma{T(x_n)}_\infty,$$
therefore $P(x_n)\longrightarrow 0$.

    (c) $\Rightarrow$ (d) Given a regular $k$-homogeneous polynomial $P \colon E \to c_0$ , we may write $P = P^+ - P^-$, where $P^+,P^- \colon E \to c_0$ are positive $k$-homogeneous polynomials. Applying (c) for $P^+$ and $P^-$ one obtains (d) easily. %If $(x_n)_n$ is a disjoint weakly $p$-summable sequence, $\lim_{n \to \infty} P(x_n)_n = 0$ and $\lim_{n \to \infty} P^-(x_n) = 0$, and consequently
%    $$ \lim_{n \to \infty} P(x_n) = \lim_{n \to \infty} P^+(x_n) - \lim_{n \to \infty} P^-(x_n) = 0.$$

    (d) $\Rightarrow$ (e) Let $(x_n)_n$ be a disjoint weakly $p$-summable sequence in $E$ and let $(P_n)_n$  be a weak* null sequence in $\mathcal{P}^r(^k E)$. Defining $P(x) = (P_i(x))_i$ for every $x \in E$ we obtain a well defined regular $k$-homogeneous polynomial  $P\colon E \to c_0$. By assumption we have $\lim\limits_{n \to \infty} P(x_n) = 0$. Thus,
    $$ |P_n(x_n)| \leq \sup_{i \in \N} |P_i(x_n)| = \norma{P(x_n)}_\infty \longrightarrow 0.  $$

    (g) $\Rightarrow$ (a) Let $(x_n)_n$ be a disjoint weakly $p$-summable sequence in $E$ and let $(x_n^\ast)_n$ be a positive weak* null sequence in $E^*$. On the one hand, letting $P_n(x) = (x_n^\ast (x))^k$ for every $x \in E$ and $n \in \N$,
    we have that $(P_n^\otimes)_n = (\otimes^k x_n^\ast)_n$ is a positive
    weak* null sequence in $\left ( \widehat{\otimes}_{s, |\pi|}^k E\right )^\ast$. On the other hand, $(\otimes^k x_n)_n$ is a disjoint sequence in $\widehat{\otimes}_{s, |\pi|}^k E$ because $(x_n)_n$ is disjoint, and by \cite[Ex.\,22, p.\,77]{alip} there exists a positive disjoint sequence $(\varphi_n)_n$ in $\left ( \widehat{\otimes}_{s, |\pi|}^k E\right )^\ast$ such that $\varphi_n \leq P_n^\otimes$ and $\varphi_n(\otimes^k x_n) = P_n^\otimes (\otimes^k x_n)$ for every $n \in \N$. Denoting by $Q_n\colon E \to \R$ the positive $k$-homogeneous polynomial whose linearization is $\varphi_n$, we get that
    $(Q_n)_n$ is a positive disjoint weak* null sequence in $\mathcal{P}^r(^k E)$. By assumption,
    $$ (x_n^\ast(x_n))^k = P_n(x_n) = P_n^\otimes (\otimes^k x_n) = \varphi_n(\otimes^k x_n) = Q_n(x_n)\longrightarrow 0. $$
    By Theorem \ref{teorema3} we conclude that $E$ has the disjoint $DP^\ast P_p$.
\end{proof}

\begin{remark}\rm
(i) The assumption of $E$ being $\sigma$-Dedekind complete in Theorem \ref{teopol2} was used only to prove (g)$\Rightarrow$(a). All other implications hold for every Banach lattice.\\
(ii) Theorem \ref{teopol1}(a) and Theorem \ref{teopol2}(a) make clear that each of the other conditions of both theorems holds for some $k \in \N$ if and only if it holds for every $k \in \N$.
\end{remark}

\bigskip

\noindent G. Botelho and V. C. C. Miranda\\
Faculdade de Matem\'atica\\%~~~~~~~~~~~~~~~~~~~~~~Instituto de Matem\'atica e Estat\'istica\\
Universidade Federal de Uberl\^andia\\%~~~~~~~~ Universidade de S\~ao Paulo\\
38.400-902 -- Uberl\^andia -- Brazil\\%~~~~~~~~~~~~ 05.508-090 -- S\~ao Paulo -- Brazil\\
e-mails: botelho@ufu.br, colferaiv@gmail.com % ~~~~~~~~~~~~~~~~~~~~~~~~~e-mail: veronica@ime.usp.br.

\medskip

\noindent J. L. P. Luiz\\
Instituto Federal do Norte de Minas Gerais\\
Campus de Ara\c cua\'i\\
39.600-000 -- Ara\c cua\'i -- Brazil\\
e-mail: lucasvt09@hotmail.com


\begin{thebibliography}{99}\small

%\bibitem{albiac} F. Albiac and N. J. Kalton, \textit{Topics in Banach Space Theory}, Springer, New York, (2006).

\bibitem{andrews} K. T. Andrews, \textit{Dunford-Pettis sets in the space of Bochner integrable functions}, Math. Ann. 241 (1979), 35-41.

\vspace*{-0.5em}

\bibitem{ali} M. Alikhani, {\it Some characterizations of almost $p$-convergent operators and applications}, J. Math. Anal. Appl. 526 (2023), no. 2, Paper No. 127316, 17 pp.

\vspace*{-0.5em}


\bibitem{alip} C. Aliprantis and O. Burkinshaw, \textit{Positive Operators}, Springer, Dordrecht, 2006.

\vspace*{-0.5em}

\bibitem{ardakani2} H. Ardakani, K. Amjadi, {\it On the positively limited $p$-Schur property in Banach lattices}, Ann. Funct. Anal. 14 (2023), no.4, Paper No. 68, 14 pp.

\vspace*{-0.5em}

\bibitem{ardachen} H. Ardakani, J. X. Chen, {\it Positively limited sets in Banach lattices}, J. Math. Anal. Appl. 526 (2023), no. 1, Paper No. 127220, 12 pp.

\vspace*{-0.5em}

\bibitem{ardakani} H. Ardakani, Kh. Taghavinejad, {\it The strong limited $p$-Schur property in Banach lattices}, Oper. Matrices 16 (2022), no. 3, 811-825.

%\vspace*{-0.5em}

%\bibitem{bombal} F. Bombal and M. Fernández, {\it Polynomial properties and symmetric tensor product of Banach spaces},
%Arch. Math. 74 (2000), 40-49.

\vspace*{-0.5em}

\bibitem{botgarmir} G. Botelho, L. A. Garcia, V. C. C. Miranda, {\it Disjoint $p$-convergent operators and their adjoints}, arXiv:2309.037752v2, 2023.



\vspace*{-0.5em}

\bibitem{jlucaspams}  G. Botelho, J. L. P. Luiz, {\it The positive polynomial Schur property in Banach lattices}, Proc. Amer. Math. Soc. 149 (2021), no. 5, 2147–2160.

\vspace*{-0.5em}

\bibitem{borwein} J. Borwein, M. Fabian and J. Vanderwerff, {\it Characterizations of Banach spaces via convex
and other locally Lipschitz functions}, Acta Math. Vietnam. 22 (1997), 53-69.

\vspace*{-0.5em}

\bibitem{boumoussa}  K. Bouras and M. Moussa, {\it Banach lattices with weak Dunford-Pettis property}, Inter. J. Inf. Math. Sci. 6 (3) (2010), 203-207.

\vspace*{-0.5em}

\bibitem{bouras} K. Bouras, {\it Almost Dunford–Pettis sets in Banach lattices}, Rend. Circ. Mat. Palermo (2) 62 (2013), no. 2, 227–236.

%\vspace*{-0.5em}

%\bibitem{buwong} Q. Bu, Ngai-Ching Wong, Some geometric properties inherited by the positive tensor products of atomic Banach lattices, Indagationes Mathematicae, Volume 23, Issue 3, 2012, Pages 199-213.

\vspace*{-0.5em}

\bibitem{bubuskes} Q. Bu and G. Buskes, {\it Polynomials on Banach lattices and positive tensor products}, J. Math.
Anal. Appl. 388 (2012), 845–862.

\vspace*{-0.5em}

\bibitem{cargalou} H. Carri\'on, P. Galindo and M. L. Lourenço, {\it A stronger Dunford-Pettis property}, Studia Math. 184 (2008), no. 3, 205–216.

\vspace*{-0.5em}

\bibitem{castgon}  J. M. F. Castillo, M. González,
{\it On the Dunford-Pettis property in Banach spaces}, Acta Univ. Carolin. Math. Phys. 35 (1994), no. 2, 5-12.


\vspace*{-0.5em}

\bibitem{chen} J. X. Chen, Z. L. Chen, G. X. Ji, {\it Almost limited sets in Banach lattices}, J. Math. Anal. Appl. 412 (2014), no. 1, 547–553.


\vspace*{-0.5em}

\bibitem{chenli} J. X. Chen, L. Li, {\it On a question of Bouras concerning weak compactness of almost Dunford-Pettis sets},
Bull. Aust. Math. Soc. 92 (2015), no. 1, 111–114.

\vspace*{-0.5em}



\bibitem{diestel} J. Diestel, H. Jarchow and A. Tonge, \textit{Absolutely Summing Operators}, Cambridge Stud. Adv. Math. 43, Cambridge Univ. Press, Cambridge, (1995).

\vspace*{-0.5em}

\bibitem{dineen} S. Dineen, \textit{Complex Analysis in Infinite Dimensional  Spaces}, Springer, London, 1999.

\vspace*{-0.5em}

\bibitem{fabian} M. Fabian, P. Habala, P. H\'ajek, V. Montesinos and V. Zizler, \textit{Banach Space Theory. The basis for linear and nonlinear analysis}. CMS Books in Mathematics/Ouvrages de Math\'ematiques de la SMC, Springer, New York, 2011.

\vspace*{-0.5em}

\bibitem{farjohn} J. Farmer, W. B. Johnson, {\it Polynomial Schur and polynomial Dunford–Pettis properties}, Banach Spaces (Mérida, 1992), 95-105, Contemp. Math., 144, Amer. Math. Soc., Providence, RI, 1993.

\vspace*{-0.5em}

\bibitem{fouzee} J. H. Fourie and E. D. Zeekoei, \textit{On weak-star $p$-convergent operators}, Quaest. Math. 40 (2017), no. 5, 563-579.

\vspace*{-0.5em}

\bibitem{galmir} P. Galindo, V. C. C. Miranda, {\it Grothendieck-type subsets of Banach lattices}, J. Math. Anal. Appl. 506 (2022), no. 1, Paper No. 125570, 14 pp.

\vspace*{-0.5em}

\bibitem{galmir2} P. Galindo, V. C. C. Miranda, {\it Some properties of $p$-limited sets}, Proc. Amer. Math. Soc., to appear. DOI: https://doi.org/10.1090/proc/16573.% \\
%Este trabalho não foi pro arxiv}

%\vspace*{-0.5em}

%\bibitem{hmichane} H’michane, J., Hafidi, N. and Zraoula, L. On the class of disjoint limited completely continuous operators. Positivity 22, 1419–1431 (2018).

\vspace*{-0.5em}

\bibitem{libu} Y. Li, Q. Bu, {\it Majorization for compact and weakly compact polynomials on Banach lattices}, Positivity and Noncommutative Analysis, 339-348, Trends Math., Birkhäuser/Springer, Cham [2019).

\vspace*{-0.5em}

\bibitem{loane} J. Loane, {\it Polynomials on Riesz Spaces}, Doctoral Thesis, National University of Ireland, Galway (2007).

\vspace*{-0.5em}

\bibitem{loumir} M. L. Louren\c co, V. C. C. Miranda, {\it A note on the Banach lattice $c_0( \ell_2^n)$ and its dual}, Carpathian Math. Publ. 15 (2023), no. 1, 270-277.

\vspace*{-0.5em}

\bibitem{loumir2} M. L. Louren\c co, V. C. C. Miranda, {\it The Property (D) and the almost limited completely continuous operators}, Anal. Math. 49 (2023), no. 1, 270--277.

\vspace*{-0.5em}

\bibitem{mach} N. Machrafi, K. El Fahri, M. Moussa, B. Altin, {\it A note on weak almost limited operators}, Hacet. J. Math. Stat. 48 (2019), no. 3, 759-770.

\vspace*{-0.5em}

\bibitem{meyer} P. Meyer-Nieberg, \textit{Banach Lattices}, Springer-Verlag, 1991.




\vspace*{-0.5em}

\bibitem{pel} A. Pełczy\'nski, {\it On weakly compact polynomial operators on $B$-spaces with Dunford–Pettis property}, Bull. Acad. Polon. Sci. S\'er. Sci. Math. Astronom. Phys. 11 (1963), 371–378.

%\vspace*{-0.5em}

%%\bibitem{rabiger} Räbiger, F.: Beiträge zur Strukturtheorie der Grothendieck-Räume. In: Sitzungsberichte der Heidelberger Akademie der Wissenschaften, Mathematisch-Naturwissenschaftliche Klasse, pp. 83–158. Springer, Berlin (1985)

\vspace*{-0.5em}

\bibitem{ryan} R. A. Ryan, {\it Dunford–Pettis properties}, Bull. Acad. Polon. Sci. S\'er. Sci. Math. 27 (1979), no. 5, 373–379.

\vspace*{-0.5em}

\bibitem{shiwangbu} Z. Shi, Y. Wang, Q. Bu, {\it Polynomial versions of almost Dunford-Pettis sets and almost limited sets in Banach lattices}, J. Math. Anal. Appl. 485 (2020), no. 2,  123834, 9 pp.

\vspace*{-0.5em}

\bibitem{stegall} C. Stegall, {\it Duals of certain spaces with the Dunford-Pettis property}, Notices Amer. Math. Soc. 19 (1972) 799. %\textcolor{blue}{(Não achei essa referência. É só uma página mesmo?)} \textcolor{red}{A referência está assim em outros trabalhos, ver citações de \cite{cargalou} e \cite{castgon}. Alguns lugares citam A799 ao invés de 799.}

\vspace*{-0.5em}

\bibitem{zeefou} E. D. Zeekoei, J. H. Fourie, {\it On $p$-convergent operators on Banach lattices}, Acta Math. Sin. (Engl. Ser.) 34 (2018), n0.5, 873–890.

\vspace*{-0.5em}

\bibitem{wangshibu} Y. Wang, Z. Shi, Q. Bu, {\it Polynomial versions of weak Dunford–Pettis properties in Banach lattices}, Positivity 25 (2021), no. 5,  1685-1698.

\vspace*{-0.5em}

\bibitem{wnukdual} W. Wnuk, {\it On the dual positive Schur property in Banach lattices}, Positivity 17 (2013), 759-773.


\end{thebibliography}
\end{document}